\documentclass[oneside,a4paper,10pt]{amsart} 
\usepackage{enumitem}
\usepackage{amsmath,amssymb,amsthm}
\usepackage[T1]{fontenc}
\usepackage[utf8]{inputenc}
\usepackage[english]{babel}
\usepackage{url}
\usepackage[hidelinks]{hyperref}
\usepackage{geometry}
 
\newcommand*{\R}{\mathbb{R}}
\newcommand{\abs}[1]{\ensuremath{\left\lvert #1 \right\rvert}}

\newcommand{\halmaz}[1]{\left\{\,#1\,\right\}}

\theoremstyle{plain}
\newtheorem{thm}{Theorem}
\newtheorem{lemma}[thm]{Lemma}
\newtheorem{remark}[thm]{Remark}

\theoremstyle{definition}
\newtheorem{problem}{Problem} 

\title[The Lie group of the general Li\'{e}nard-type equation]{The Lie symmetry group of the general Li\'{e}nard-type equation}

\author[\'A.~Figula]{\'Agota Figula}
\email{\url{figula@science.unideb.hu}}
\author[G.~Horv\'ath]{G\'abor Horv\'ath}
\email{\url{ghorvath@science.unideb.hu}}
\author[T.~Milkovszki]{Tam\'as Milkovszki}
\email{\url{milkovszki@science.unideb.hu}}
\author[Z.~Muzsnay]{Zolt\'an Muzsnay}
\email{\url{muzsnay@science.unideb.hu}}
\address{University of Debrecen, Institute of Mathematics, Pf.~400,
  Debrecen, 4002, Hungary}

\thanks{ The research was supported by the European Union's Seventh Framework
  Programme (FP7/2007-2013) under grant agreement no.~318202.  The first,
  third and fourth authors were partially supported by the European Union's
  Seventh Framework Programme (FP7/2007-2013) under grant agreement
  no.~317721.  The second author was partially supported by the Hungarian
  Scientific Research Fund (OTKA) grant no.~K109185 and grant
  no.~FK124814. The first and fourth authors were partially supported by the
  EFOP-3.6.1-16-2016-00022 project.  This projects have been supported by the
  European Union, co-financed by the European Social Fund.}



\begin{document}

\begin{abstract}
  We consider the general Li\'{e}nard-type equation
  $\ddot{u} = \sum_{k=0}^n f_k \dot{u}^k$ for $n\geq 4$.  This equation
  naturally admits the Lie symmetry $\frac{\partial}{\partial t}$.  We
  completely characterize when this equation admits another Lie symmetry, and
  give an easily verifiable condition for this on the functions
  $f_0, \dots , f_n$.  Moreover, we give an equivalent characterization of
  this condition.  Similar results have already been obtained previously in
  the cases $n=1$ or $n=2$.  That is, this paper handles all remaining cases
  except for $n=3$.
\end{abstract}
\maketitle

\section{Introduction}\label{sec:intro}

In this paper we consider the Li\'{e}nard-type second order ordinary
differential equation
\begin{equation}
  \label{EQ:01}
  \ddot{u}=\sum_{k=0}^n f_k(u) \dot{u}^k, 
\end{equation}
where the dot denotes differentiating by the independent variable $t$
representing the time.  Equation \eqref{EQ:01} is a special case of the
Levinson--Smith-type equation
\begin{math}
  \ddot{u} = g_1 \left( u, \dot{u} \right) \dot{u} +g_0(u)
\end{math}
\cite[G.3, p.~198--199]{Mickens1981}, for which existence and uniqueness of
a limit cycle have been established under certain conditions
\cite{Levinson1943, LevinsonSmith1942}.

Equation \eqref{EQ:01} is a common generalization of the Rayleigh-type
equation $\ddot{u}+F(\dot{u})+u=0$ when $F$ is a polynomial, the classical Li\'{e}nard-type equation
$\ddot{u}=f_1(u) \dot{u} + f_0(u)$, and the quadratic Li\'{e}nard-type
equation $\ddot{u}=f_2(u) \dot{u}^2 + f_1(u) \dot{u} + f_0(u)$.  These
equations come up quite often in Physics or Biology.  Rayleigh-type systems
play an important role in the theory of sound \cite{Rayleigh1945} or in the
theory of non-linear oscillations \cite[Chapter~2.2.4]{HandbookODE}.
Classical Li\'{e}nard-type equations arise in the model of the van der Pol
oscillator applied in physical and biological sciences
\cite{vanderPol1927}, but electric activity of the heart rate
\cite{vanderPolvanderMark1928} or nerve impulses are modelled by a
Li\'{e}nard-type model, as well \cite{Fitzhugh1961, Nagumo1962} or
\cite[Chapter 7]{Jones}.  In \cite{NucciSanchini2015} the population of
Easter Island is modelled, and the system of differential equations is then
reduced to a second order quadratic Li\'{e}nard-type equation.  One can
even find applications in economy \cite{Faria2008, Goodwin1975, Kaldor1940,
  Kalecki1937}.

Symmetry analysis is a very useful tool developed to understand and solve
differential equations. Several examples come from Physics (see
e.g.~\cite{vinogradov, Olver} for comprehensive studies on the topic), and
an increasing number of examples from Biology (see
e.g.~\cite{Nucci1Biology, Nucci, Nucci2Biology, NucciSanchini2015}).
Finding some symmetries for a differential equation can be used to derive
an appropriate change of coordinates which then helps to eliminate the
independent variables or to decrease the order of the system.

In many cases mentioned above (e.g.\ the Fitzhugh--Nagumo model
\cite{Fitzhugh1961, Nagumo1962} or the model for the population of Easter
Island \cite{NucciSanchini2015}) the model is based on a first order system
of two equations equivalent to a second order Li\'{e}nard equation.  In
such a case, one might benefit to consider the equivalent second order
system, which would only admit a finite dimensional Lie symmetry group
instead of an infinite one.  If this Lie group is at least two-dimensional,
then pulling the symmetries back to the original system could yield two
independent symmetries of the original system, and solutions can be
determined by quadratures.  This method has been applied successfully in
several situations in the past (see e.g.~\cite{Nucci1Biology,
  Nucci2Biology, NucciSanchini2015} for some recent examples in Biology).
This motivates to study the Lie symmetries of \eqref{EQ:01}.

Pandey, Bindu, Senthilvelan and Lakshmanan \cite{Pandey2009-1, Pandey2009-2} considered the classical Li\'{e}nard equation
\begin{equation}
  \label{eq:pandey}
  \ddot{u}=f_1(u) \dot{u} + f_0(u), 
\end{equation}
where $f_1$ and $f_0$ are arbitrary, infinitely many times differentiable
functions.  They classified when \eqref{eq:pandey} has a 1, 2, 3, or 8
dimensional Lie symmetry group depending on $f_0$ and $f_1$.  Then Tiwari,
Pandey, Senthilvelan and Lakshmanan \cite{Tiwari2013, Tiwari2014}
classified the dimension of the Lie symmetry group of quadratic
Li\'{e}nard-type equations without $\dot{u}$ term, and then more generally
\cite{Tiwari2015} the mixed quadratic Li\'{e}nard-type equation
\begin{equation}\label{eq:tiwari}
\ddot{u}=f_2(u) \dot{u}^2+ f_1(u) \dot{u} + f_0(u), 
\end{equation}
where $f_0$, $f_1$ and $f_2$ are arbitrary, infinitely many times
differentiable functions.  Further, Paliathanasis and Leach
\cite{PaliathanasisLeach2016} showed how one can simplify \eqref{eq:tiwari}
by removing $f_2$ from \eqref{eq:tiwari} in the case $f_1=0$.  The question
naturally arises: what are the Lie symmetries if the right-hand side of
\eqref{eq:tiwari} is a higher order polynomial in $\dot{u}$?

In this paper we consider \eqref{EQ:01} for $n\geq 4$ and for
differentiable functions $f_k$ depending only on $u$, and not on $t$.
Note, that equation \eqref{EQ:01} is autonomous, therefore the tangential
Lie algebra $\mathcal{L}$ of the Lie group of all its symmetries always
contains the $1$-dimensional subalgebra generated by the vector field
$\frac{\partial}{\partial t}$.  Determining another generator of
$\mathcal{L}$ would then lead to a solution by quadratures of
\eqref{EQ:01}, and of any first order system equivalent to it.  In
Theorem~\ref{thm:main} (see Section~\ref{sec:alg} for details) we completely characterize the case when
\eqref{EQ:01} admits a more than 1 (in fact, 2) dimensional symmetry group.
In particular, we give conditions (\ref{eq:f_nnotzero}--\ref{eq:0-nuj}) such that the symmetry group is 2 dimensional if and only if these conditions hold. 

Here, conditions (\ref{eq:f_nnotzero}--\ref{eq:F}) are natural, but the meaning
of the system \eqref{eq:0-nuj} 
seems less intuitive, even though the system
\eqref{eq:0-nuj} 
is easily verifiable for a particular choice of $F$.  In
Theorem~\ref{thm:conditions} (see Section~\ref{sec:conditions} for details) we provide a necessary and sufficient
condition for $f_0, \dots , f_n$ to satisfy \eqref{eq:0-nuj}.  It turns out
that $f_0, \dots , f_n$ satisfy \eqref{eq:0-nuj} if and only if they are
expressible by $F$ and some constants.

\bigskip

\section{The symmetry condition}\label{sec:symcond}

We formulate the symmetry condition for \eqref{EQ:01} in this section.
Consider equation~\eqref{EQ:01} on the plane $(t,u)$, where $t$ is the
independent variable, and $u$ is the dependent variable.  Further, the
computations will be slightly easier if we consider the right-hand side as
an infinite sum
$\sum_k f_k \dot{u}^k = \sum_{k = -\infty}^{\infty} f_k \dot{u}^k$, where
$f_k=0$ if $k < 0$ or $k>n$.

The general form of an infinitesimal generator of a symmetry of
\eqref{EQ:01} has the form
\begin{equation}
\label{eq:inf_gen}
X=\xi(t,u)\frac{\partial}{\partial t}+\eta(t,u)\frac{\partial}{\partial u}.
\end{equation}
Let $D$ denote the total derivation by $t$, that is
$D\xi = \xi_t + \dot{u} \xi_u$, $D\eta = \eta_t + \dot{u} \eta_u$.  We use
the convention of writing partial derivatives into the lower right index.
Then the first prolongation of $X$ is
\begin{displaymath}
  X^{1}=\xi\frac{\partial}{\partial t}+\eta\frac{\partial}{\partial
    u}+(D\eta-\dot{u}D\xi)\frac{\partial}{\partial \dot{u}} =
  \xi\frac{\partial}{\partial t}+\eta\frac{\partial}{\partial u} + \Bigl(
    \eta_t + \left( \eta_u - \xi_t \right) \dot{u} - \xi_u \dot{u}^2
  \Bigr) \frac{\partial}{\partial \dot{u}}.
\end{displaymath}
Further, let
\begin{displaymath}
  S^1 =\frac{\partial}{\partial t}+\dot{u}\frac{\partial}{\partial u}+ \left( \sum_{k} f_k \dot{u}^k \right) \frac{\partial}{\partial \dot{u}},
\end{displaymath}
be the spray corresponding to the differential equation~\eqref{EQ:01}.  The
vector field \eqref{eq:inf_gen} is an infinitesimal symmetry of
\eqref{EQ:01} if and only if its first prolongation $X^1$ satisfies the Lie
bracket condition
\begin{equation}
  \label{symmetrycondition} 
  [X^{1}-\xi S^{1},S^{1}]= 0  
\end{equation}
on the space $(t,u,\dot u)$ (cf.~\cite[Chapter 4, \S
3]{vinogradov}). 
Substituting $X^1$ and $S^1$ into \eqref{symmetrycondition} we obtain
\begin{align*}
  0&=  \left[ X^{1} - \xi S^{1}, S^{1} \right]
  \\
   &= \Bigl( \left(\eta - \xi \dot{u} \right) \Bigl( \sum_{k} f_k \dot{u}^k \Bigr)_u
     + \Bigl( \eta_t + \left( \eta_u - \xi_t \right) \dot{u} - \xi_u \dot{u}^2  - \xi \bigl( \sum_{k} f_k \dot{u}^k \bigr) \Bigr) \Bigl( \sum_{k} f_k \dot{u}^k \Bigr)_{\dot{u}} 
  \\
   &\qquad -\Bigl( \eta_t + \left( \eta_u - \xi_t \right) \dot{u} - \xi_u \dot{u}^2 -\xi \Bigl( \sum_{k} f_k \dot{u}^k \Bigr)\Bigr)_t 
     -\dot{u} \Bigl( \eta_t + \left( \eta_u - \xi_t \right) \dot{u} - \xi_u \dot{u}^2 - \xi \Bigl( \sum_{k} f_k \dot{u}^k \Bigr)\Bigr)_u 
  \\
   &\qquad  -\Bigl( \sum_{k} f_k \dot{u}^k \Bigr) \Bigl( \eta_t + \left( \eta_u - \xi_t \right) \dot{u} - \xi_u \dot{u}^2 -\xi \Bigl( \sum_{k} f_k \dot{u}^k \Bigr)\Bigr)_{\dot{u}} \Bigr)
     \frac{\partial}{\partial \dot{u}} 
  \\
   &= \Bigl( -\eta_{tt} + \left( \xi_{tt} - 2 \eta_{tu} \right) \dot{u} + \left( 2 \xi_{ut} - \eta_{uu} \right) \dot{u}^2 + \xi_{uu} \dot{u}^3 \\
   &\qquad + \sum_k \left( f_k' \eta +(k+1)f_{k+1} \eta_t + (k-1)f_k \eta_u + (2-k)f_k \xi_t +(4-k)f_{k-1} \xi_u \right) \dot{u}^k \Bigr)
     \frac{\partial}{\partial \dot{u}},
\end{align*}
therefore the symmetry condition is
\begin{equation}
\label{EQ:0}
\begin{alignedat}{1}
  - \eta_{tt} + f_0'\eta + f_1\eta_t -f_0 \eta_u + 2 f_0 \xi_t \hphantom{)
    \cdot \dot{u}}&
  \\
  + \bigl( \xi_{tt} - 2 \eta_{tu} + f_1'\eta + 2 f_2 \eta_t + f_1\xi_t + 3
  f_0 \xi_u \bigr) \cdot \dot{u} &
  \\
  + \bigl(2\xi_{tu}-\eta_{uu} + f_2' \eta + 3 f_3 \eta_t + f_2 \eta_u + 2
  f_1 \xi_{u} \bigr) \cdot \dot{u}&^2
  \\
  + \bigl( \xi_{uu} + f_3' \eta + 4 f_4 \eta_t + 2 f_3 \eta_u -f_3 \xi_t +
  f_2 \xi_u \bigr) \cdot \dot{u}&^3
  \\
  + \textstyle{\sum_{\, k = 4}^{\, n-1}} \bigl( f_k' \eta + (k+1)f_{k+1}
  \eta_t + (k-1)f_k \eta_u + (2-k)f_k \xi_t + (4-k)f_{k-1} \xi_u \bigr)
  \cdot \dot{u}&^k
  \\
  + \bigl( f_n' \eta + (n-1)f_n \eta_u + (2-n)f_n \xi_t + (4-n)f_{n-1}
  \xi_u \bigr) \cdot \dot{u}&^n
  \\
  + (3-n)f_{n} \xi_u \cdot \dot{u}&^{n+1} & =0.
\end{alignedat}
\end{equation}

\section{Lie symmetry algebra}\label{sec:alg}

We consider \eqref{EQ:01} for $n\geq 4$ and for differentiable functions
$f_k$ depending only on $u$, and not on $t$. In Theorem~\ref{thm:main} we
completely characterize the case when \eqref{EQ:01} admits more then 1 (in
fact, 2) dimensional symmetry group. We prove following
\begin{thm}
  \label{thm:main}
  Consider the equation~\eqref{EQ:01} for some $n\geq 4$ and for
  $f_0, \dots , f_n \colon I \to \R$ being differentiable functions such
  that $f_n$ is not constant zero on the open interval $I \subseteq \R$.
  Then the Lie symmetry algebra $\mathcal{L}$ of \eqref{EQ:01} is exactly 1
  dimensional, unless there exists a constant $a \in \R$, an open interval
  $U \subseteq I$, and a three-times differentiable function
  $F \colon U \to \R$ such that for all $u \in U$ we have
  \begin{align}
  \label{eq:f_nnotzero}
    &f_n(u)  \neq 0.\\
    \label{eq:Fnotzero}
    &F(u) \neq 0, \\
    \label{eq:F}
    & F'(u) = \abs{f_n(u)}^{\frac{1}{n-1}}, 
  \end{align}
  and with the notation
  \begin{math}
    g(u) = \frac{(n-2)F(u)}{(n-1)F'(u)}
  \end{math}
  the following hold:
  \begin{subequations}
    \label{eq:0-nuj}
    \begin{alignat}{5}
      \label{eq:0uj}
       -a^2g + f_0'g + af_1g + (-1)f_0 g' + 2 f_0 &=0,
      \\
      \label{eq:1uj}
       a\left( 1 - 2 g' \right) + f_1'g + 2a f_2 g +  f_1 &=0,
      \\
      \label{eq:2uj}
       -g'' + f_2' g + 3 a f_3 g + f_2 g'  &=0,
      \\
      \label{eq:kuj}
      \hphantom{3 \leq k \leq n-1.} f_k' g + (k+1) a f_{k+1} g +
      (k-1)f_k g' + (2-k) f_k &=0, \qquad 3 \leq k \leq n-1.
    \end{alignat}
  \end{subequations}
  Further, if both $F_1$ and $F_2$ satisfy conditions
  (\ref{eq:f_nnotzero}--\ref{eq:0-nuj}), then $F_1 = F_2$.
\end{thm} 

\begin{remark}
  (Generator of the symmetry algebra $\mathcal{L}$ of \eqref{EQ:01}.) In
  case conditions (\ref{eq:f_nnotzero}--\ref{eq:0-nuj})
  \begin{enumerate}
  \item do not hold, then the symmetry algebra $\mathcal{L}$ is generated
    by $\frac{\partial}{\partial t}$,
  \item hold, then the 2 dimensional Lie symmetry algebra $\mathcal{L}$ is
    generated 
    \begin{enumerate}
    \item [i)] by $\frac{\partial}{\partial t}$ and
      $t \frac{\partial}{\partial t} + g(u) \frac{\partial}{\partial u}$ if
      $a= 0$, or 
    \item [ii)] by $\frac{\partial}{\partial t}$ and
      $e^{at} \frac{\partial}{\partial t} + ae^{at}g(u)
      \frac{\partial}{\partial u}$ if $a\neq0$.
    \end{enumerate}
  \end{enumerate}
\end{remark}

\begin{proof}
  The left-hand side of \eqref{EQ:0} has to be zero for all $(t,u,\dot u)$.
  As $\xi$, $\eta$, $f_k$ ($0\leq k\leq n$) do not depend on $\dot u$,
  \eqref{EQ:0} is a polynomial in $\dot u$.  Thus, \eqref{EQ:0} holds if
  and only if each of its coefficients is zero.  Since $f_n$ is not
  constant 0 on the interval $I$, there exists an open interval
  $U' \subseteq I$ such that $f_n(u)\neq 0$ for $u \in U'$.  Thus, from the
  coefficient of $\dot{u}^{n+1}$ by $n\geq 4$ we obtain
  \begin{equation*}
    \xi_u = 0.
  \end{equation*}
  In particular, $\xi$ only depends on $t$ and not on $u$.  Substituting
  $\xi_u=0$ into \eqref{EQ:0} and considering the coefficients, we obtain
  that
  \begin{subequations} 
    \label{eq:0-n}
    \begin{alignat}{5}
      \label{eq:0}
      -\eta_{tt} + f_0'\eta + f_1\eta_t +(-1) f_0 \eta_u + 2 f_0 \xi_t & =0,
      \\
      \label{eq:1}
      \left( \xi_{tt} - 2 \eta_{tu} \right) + f_1'\eta + 2 f_2 \eta_t +
      f_1\xi_t & =0,
      \\
      \label{eq:2}
      -\eta_{uu} + f_2' \eta + 3 f_3 \eta_t + f_2 \eta_u & =0,
      \\
      \label{eq:k}
      f_k' \eta + (k+1)f_{k+1} \eta_t + (k-1)f_k \eta_u + (2-k)f_k \xi_t
      & =0, \qquad 3 \leq k \leq n-1
      \\
      \label{eq:n}
      f_n' \eta + (n-1)f_n \eta_u + (2-n)f_n \xi_t & =0.
    \end{alignat}
  \end{subequations} 
  In the following we analyze the system
  \eqref{eq:0-n}. 
  Note, that $\xi = c$, $\eta = 0$ (for any $c \in \R$) satisfies
  \eqref{eq:0-n}. 
  Further, if $\eta = 0$ then from \eqref{eq:n} we have $\xi_t = 0$ and
  $\xi = c$ for some $c \in \R$.  Thus, in the following we assume that
  $\eta$ is not constant 0.

  Consider \eqref{eq:n} first.  Now, $f_n$ is nonzero on the open interval
  $U'$, therefore either $f_n(u)>0$ for all $u \in U'$ or $f_n(u)<0$ for
  all $u \in U'$.  Choose $\varepsilon \in \halmaz{1, -1}$ such that
  $\varepsilon f_n(u) > 0$ for all $u \in U'$, that is
  $\abs{f_n} = \varepsilon f_n$.  Then we have
  \begin{displaymath}
    \abs{f_{n}}' \eta + (n-1) \abs{f_{n}} \eta_u + (2-n)\abs{f_{n}} \xi_t =0.
  \end{displaymath}
  Multiplying by $\frac{1}{n-1}\abs{f_{n}}^{\frac{2-n}{n-1}}$ yields
  \begin{displaymath}
    \frac{\abs{f_{n}}' \cdot \abs{f_{n}}^{\frac{2-n}{n-1}} }{(n-1)} \eta
    + \abs{f_{n}}^{\frac{1}{n-1}} \eta_u = \frac{n-2}{n-1}
    \abs{f_{n}}^{\frac{1}{n-1}}  \xi_t. 
  \end{displaymath}
  Here, the left-hand side is the $u$-derivative of
  $\abs{f_{n}}^{\frac{1}{n-1}} \eta$, hence
\begin{align*}
  \left( \abs{f_{n}}^{\frac{1}{n-1}} \eta \right)_u &= \frac{n-2}{n-1} \abs{f_{n}}^{\frac{1}{n-1}}  \xi_t, \\
  \abs{f_{n}}^{\frac{1}{n-1}} \eta &= \frac{n-2}{n-1} \xi_t \int
                                     \abs{f_{n}}^{\frac{1}{n-1}} du,
  \\
  \eta &= \xi_t \frac{\left(n-2\right) \int \abs{f_{n}}^{\frac{1}{n-1}} du}{\left( n-1 \right) \abs{f_{n}}^{\frac{1}{n-1}}}.
\end{align*}
Let $F \colon U' \to \R$ be a function defined as in \eqref{eq:F}, that is 
\begin{align*}
  F'(u) &:= \abs{f_{n}(u)}^{\frac{1}{n-1}}, 
          \intertext{%
          and let $g$ be defined as in Theorem \ref{thm:main}, that is}
          g(u) &:= \frac{\left(n-2\right) F(u)}{\left( n-1 \right) F'(u)}. 
\end{align*}
Thus we obtained that
\begin{equation}
\label{eq:eta}
\eta = g(u) \cdot \xi_t(t). 
\end{equation}
Note, that $g(u)$ cannot be constant 0 on $U'$ (otherwise both $F$ and
$F' = \abs{f_{n}}^{\frac{1}{n-1}}$ were constant 0), thus there exists an
open interval $U \subseteq U'$ such that $F(u) \neq 0$ for all $u \in U$,
and hence $g(u)\neq 0$ ($u \in U$).  By substituting \eqref{eq:eta} into
\eqref{eq:k} we have
\begin{displaymath}
  - (k+1)f_{k+1} g \xi_{tt} = \left( f_k' g + (k-1)f_k g' +(2-k)f_k \right)
  \xi_t.
\end{displaymath}
In particular, for $k=n-1$ we have $f_n(u) g(u) \neq 0$ for all $u \in U$,
thus
\begin{displaymath}
  \xi_{tt} = - \frac{f_{n-1}' g + (n-2)f_{n-1} g'
    + (3-n)f_{n-1}}{n f_n g }\xi_t . 
\end{displaymath}
Now, $f_n$, $f_{n-1}$, and $g$ only depend on $u$, and $\xi_{tt}$ and
$\xi_t$ only depend on $t$.  Therefore there exists $a \in \R$ such that
\begin{align}
\label{eq:a}
a & = - \frac{f_{n-1}' g + (n-2)f_{n-1} g' + (3-n) f_{n-1}}{n f_n g }, \\
\label{eq:xi}
  \xi_{tt} &= a \xi_t,
             \intertext{%
             or else $\xi_t = 0$ implying $\eta = 0$ by \eqref{eq:eta}, a contradiction. 
             Thus,
             }
\label{eq:xitfinal}
\xi_t &= c e^{at}, \\
\label{eq:etafinal}
\eta &= c e^{at} g(u)
\end{align}
for some $c \in \mathbb{R}$.  Substituting \eqref{eq:xitfinal} and
\eqref{eq:etafinal} into \eqref{eq:0-n}, 
one obtains
\begin{alignat*}{5}
  (-a^2g + f_0'g + af_1g + (-1) f_0 g' + 2 f_0 ) \cdot ce^{at} & =0,
  \\
  ( a\left( 1 - 2 g' \right) + f_1'g + 2a f_2 g + f_1) \cdot ce^{at} & =0,
  \\
  (-g'' + f_2' g + 3 a f_3 g + f_2 g' ) \cdot ce^{at} & =0,
  \\
  (f_k' g + (k+1) a f_{k+1} g + (k-1)f_k g' + (2-k) f_k) \cdot ce^{at} &
  =0, \qquad 3 \leq k \leq n-1.
\end{alignat*}
Thus, either $f_0, f_1, f_2, \dots f_n, g, F$ satisfy the conditions
(\ref{eq:f_nnotzero}--\ref{eq:0-nuj}), or else $c=0$, implying $\xi_t = 0$
and $\eta = 0$, a contradiction.

Finally, assume that $F_1$ and $F_2$ both satisfy conditions
(\ref{eq:Fnotzero}--\ref{eq:F}) on $U$.  Then let $g_1$ and $g_2$ be
defined from $F_1$ and $F_2$, and let $a_1$ and $a_2$ be defined using
\eqref{eq:a}.  Let $\xi_1$, $\xi_2$ be such that $(\xi_i)_t = e^{a_it}$,
let $\eta_i = e^{a_it} g_i(u)$, and let
$X_i = \xi_i \frac{\partial}{\partial t} + \eta_i \frac{\partial}{\partial
  u}$ for $i =1, 2$.  Further, let $\xi = \xi_1 - \xi_2$,
$\eta = \eta_1 - \eta_2$, and
$X = \xi \frac{\partial}{\partial t} + \eta \frac{\partial}{\partial u}$.
Now, if $F_1$ and $F_2$ both satisfy conditions
(\ref{eq:Fnotzero}--\ref{eq:F}) and
\eqref{eq:0-nuj}, 
then both $X_1$ and $X_2$ are elements of the Lie symmetry algebra
$\mathcal{L}$, and thus
$X = X_1 - X_2 = \left(\xi_1 - \xi_2\right) \frac{\partial}{\partial t} +
\left( \eta_1 - \eta_2 \right) \frac{\partial}{\partial u}\in \mathcal{L}$,
as well.  However, we have proved that if either $\xi_t \neq 0$ or
$\eta \neq 0$, then they are of the form \eqref{eq:xitfinal} and
\eqref{eq:etafinal}.  Thus, $e^{a_1t}-e^{a_2t}$ is of the form $c e^{at}$
for some $a, c\in \R$, that is $a_1 = a_2$ and $c = 0$.  Then $\xi_t = 0$,
implying $\eta = 0$ by \eqref{eq:eta}, which yields $g_1 = g_2$.  Finally,
by $F_1'=F_2'$, the definition of $g$ immediately implies
\[
0 = g_1 - g_2 = \frac{n-2}{(n-1) F_1'}\left( F_1 - F_2 \right),
\]
and hence $F_1 = F_2$. 
This finishes the proof of Theorem~\ref{thm:main}.
\end{proof}

\section{Equivalent description of the conditions}
\label{sec:conditions}

In this paragraph we provide a necessary and sufficient condition for
$f_0, \dots , f_n$ to satisfy \eqref{eq:0-nuj}. We have the following 

\begin{thm}\label{thm:conditions}
  Let $U \subseteq \R$ be an open interval,
  $f_0, \dots , f_n, F \colon U \to \R$ be functions satisfying
  (\ref{eq:f_nnotzero}--\ref{eq:F}), $g$ as introduced in Theorem
  \ref{thm:main} and let $\varepsilon, \nu \in \halmaz{1,-1}$, $a \in \R$
  be real constants such that $\abs{f_n} = \varepsilon f_n$,
  $\abs{F} = \nu F$.  
  Then $f_0, \dots , f_n, F, g$ satisfy
  \eqref{eq:0-nuj} 
if and only if there exist real constants $b_k$ 
and functions $A_k, B_k \colon U \to \R$ ($0 \leq k \leq n$) 
where $b_n = \varepsilon \left(\frac{n-1}{n-2} \right)^{1-n}$,
$A_n$ is constant 0, 
\begin{align}
    \label{eq:Ak}
    A_{k}(u) &= \sum \limits_{i=1}^{n-k} (-\nu)^i \binom{k+i}{i}
    a^i b_{k+i} \abs{F(u)}^{\frac{i(n-1)}{n-2}}, & & (0\leq k\leq n-1),
\end{align}
and 
\begin{subequations}
  \begin{align}
    \label{eq:B_k}
    B_k(u) &= A_k(u), & (3 \leq k \leq n) \\
    \label{eq:B_2}
    B_2(u) &= \nu A_2(u), \\
    \label{eq:B_1}
    B_1(u) &= A_1(u)-a \abs{F(u)}^{\frac{n-1}{n-2}}(2 b_2(1- \nu)+1), \\
    \label{eq:B_0}
    B_0(u) &= A_0(u)+a^2 \left( 1+ b_2(1- \nu) \right) \nu \abs{F(u)}^{\frac{2(n-1)}{n-2}},
\end{align}
\end{subequations}
such that $f_0, \dots , f_n$ are of the form
  \begin{subequations}
    \begin{align}
      \label{eq:fk}
      f_k(u) &= \left( b_k+B_k(u) \right) \cdot \left( \frac{n-1}{n-2}
               \right)^{k-1} \cdot \abs{F(u)}^{\frac{k-n }{n-2}}
               \cdot \left( F'(u) \right)^{k-1}, & (0\leq k\leq n, k\neq 2)
      \\
      \label{eq:f2}
      f_2(u) &= \left( b_2 +B_2(u) \right) \cdot \frac{n-1}{n-2} \cdot
               \frac{F'(u)}{F(u)} + \frac{F'(u)}{F(u)} - \frac{F''(u)}{F'(u)}.
    \end{align}
  \end{subequations}
             In particular, 
             if $a=0$, 
             then $B_k(u)=0$ for all $u \in U$ ($0\leq k\leq n$), 
and thus $f_0, \dots , f_n, F, g$ satisfy
  \eqref{eq:0-nuj} 
if and only if there exist real constants $b_k'$ ($0 \leq k \leq n$) such that 
  \begin{subequations}
    \begin{align*}
      f_k(u) &= b_k' \cdot \abs{F(u)}^{\frac{k-n }{n-2}}
               \cdot \left( F'(u) \right)^{k-1}, & (0\leq k\leq n, k\neq 2)
      \\
      f_2(u) &= b_2' \cdot
               \frac{F'(u)}{F(u)} - \frac{F''(u)}{F'(u)}.
    \end{align*}
  \end{subequations}

             Further,
             if $F$ is positive on $U$, 
             then
\begin{subequations}
  \begin{align}
             \notag
             B_k(u) &= A_k(u), & (2 \leq k \leq n) \\
    \notag
    B_1(u) &= A_1(u)-a \left( F(u) \right)^{\frac{n-1}{n-2}}, \\
    \notag
    B_0(u) &= A_0(u)+a^2 \left( F(u) \right)^{\frac{2(n-1)}{n-2}}.
\end{align}
\end{subequations}
\end{thm} 

First, in Section~\ref{ssec:aux} we show that $A_k$ ($0\leq k\leq n$)
defined by \eqref{eq:Ak} satisfy a recursive system of differential
equations.  The details are contained in Lemma~\ref{lem:Ak}.  Then in
Section~\ref{ssec:a=0} we consider the case $a=0$, when
\eqref{eq:0-nuj} 
results in homogeneous equations for $f_k$.  Finally, in
Section~\ref{ssec:aneq0} we prove Theorem~\ref{thm:conditions} by
considering the general case $a\neq 0$ and applying the method of variation
of parameters.

\subsection{Auxiliary functions}\label{ssec:aux}

Let us use the notations of Theorem~\ref{thm:conditions}. 

\begin{lemma} \label{lem:Ak}
Let 
$A_n(u)=0$. 
For all $0 \leq k \leq n-1$ a particular solution of the ordinary differential equation 
\begin{equation} \label{eq:Aksystem}
A_{k}(u)'=- \left( k+1 \right) a \left(\frac{n-1}{n-2}\right) \abs{F(u)}^{\frac{1}{n-2}} F(u)' \left( b_{k+1} + A_{k+1}(u) \right), 
\end{equation} 
where $b_{k+1}$ is an arbitrary real constant, 
is the function $A_{k}$ defined by \eqref{eq:Ak} in Theorem~\ref{thm:conditions}.
\end{lemma} 

\begin{proof}
We prove \eqref{eq:Ak} by induction on $m=n-k$. 
For $m=1$ we have 
\begin{align*}
A'_{n-1} &= - n a b_n \left( \frac{n-1}{n-2} \right) \abs{F}^{\frac{1}{n-2}} F', 
\intertext{%
and a particular solution is
}
A_{n-1} &= - \nu n a b_n \abs{F}^{\frac{n-1}{n-2}}. 
\end{align*}
This proves \eqref{eq:Ak} for $k=n-1$. 
Assume now, 
that \eqref{eq:Ak} holds for an integer $m=n-k$, $1 \leq m \leq n$, 
that is
\[
A_{n-m} =
\left(\sum \limits_{i=1}^{m} (- \nu)^i \binom{n-m+i}{i} a^i b_{n-m+i} \abs{F}^{\frac{i(n-1)}{n-2}} \right).  
\]
Putting this into \eqref{eq:Aksystem} for $k=n-(m+1)$
one obtains
\begin{align*}
A'_{n-(m+1)} &= -(n-m) a \left(\frac{n-1}{n-2}\right) |F |^{\frac{1}{n-2}} F' \cdot 
\left( b_{n-m} + \left( 
\sum \limits_{i=1}^{m} (- \nu)^i \binom{n-m+i}{i} a^i b_{n-m+i} \abs{F}^{\frac{i(n-1)}{n-2}} \right) \right).  
\intertext{%
By integrating, 
one can obtain a particular solution as
}
A_{n-(m+1)} &=- \nu (n-m) a b_{n-m} \abs{F}^{\frac{n-1}{n-2}} 
-\nu  \sum \limits_{i=1}^{m} (- \nu)^i \frac{n-m}{i+1} \binom{n-m+i}{i} a^{i+1} b_{n-m+i} \abs{F}^{\frac{(i+1)(n-1)}{n-2}} \\
&=- \nu (n-m) a b_{n-m} \abs{F}^{\frac{n-1}{n-2}}
+\sum \limits_{j=2}^{m+1} (- \nu)^{j}  \frac{n-m}{j} \binom{n-m-1+j}{j-1} a^{j} b_{n-m-1+j} \abs{F}^{\frac{j(n-1)}{n-2}} \\
&= \sum \limits_{i=1}^{m+1} (- \nu)^i \binom{n-(m+1)+i}{i} a^{i} b_{n-(m+1)+i} \abs{F}^{\frac{i(n-1)}{n-2}}.
\end{align*}
Hence,
\eqref{eq:Ak} holds for $k=n-(m+1)$ and by induction it holds for all integers $0 \leq k \leq n-1$. 
\end{proof}

\subsection{The homogeneous case}\label{ssec:a=0}

Assume $a=0$. 
Now, \eqref{eq:0-nuj} 
takes the form
\begin{subequations}
\label{eq:a0}
\begin{alignat}{4}
\label{eq:0a0}
 f_0'g + (-1)f_0 g' + 2 f_0 &=0,
\\
\label{eq:1a0}
    f_1'g  +  f_1 &=0,
\\
\label{eq:2a0}
  -g'' + f_2' g + f_2 g'  &=0,
\\
\label{eq:ka0}
 f_k' g + (k-1)f_k g' + (2-k) f_k  & =0, \qquad (3 \leq k \leq n-1). 
\end{alignat}
\end{subequations}
Note, 
that \eqref{eq:ka0} for $k=0, 1$ gives \eqref{eq:0a0} and \eqref{eq:1a0}. 
Now, 
$g(u) \neq 0$ for $u \in U$, 
hence the solution of \eqref{eq:ka0} is
\begin{align}
\notag
f_k &= \exp \left( \int \frac{(1-k)g'+(k-2)}{g} \right) = \exp \left( (1-k) \ln \abs{g} + (k-2) \int \frac{1}{g} \right) \\
\label{eq:solforf_kifa=0}
&= \abs{g}^{1-k} \cdot \exp \left( \left( k-2 \right) \int \frac{\left( n-1 \right) F'}{\left(n-2\right) F} \right)
= \abs{g}^{1-k} \cdot \exp \left( \frac{\left( k -2 \right)  \left( n-1 \right) }{n-2}\int \left( \ln \abs{F} \right)' \right) \\
\notag
&= b_k \cdot \abs{g}^{1-k} \cdot \abs{F}^{\frac{\left( k-2 \right) \left( n-1 \right) }{n-2}} 
= b_k \cdot \left( \frac{n-1}{n-2} \right)^{k-1} \cdot \left( F' \right)^{k-1} \cdot  \abs{F}^{\frac{k-n }{n-2}} 
\intertext{%
for some $b_k \in \R$ ($0 \leq k\leq n-1$, $k \neq 2$). 
For $k=2$ equation \eqref{eq:2a0} has the form $(-g'+f_2g)'=0$, 
thus
}
\label{eq:solforf_2ifa=0}
f_2 &= \frac{g'+b_2}{g} = \left( 1 + b_2 \frac{n-1}{n-2} \right) \cdot \frac{F'}{F} - \frac{F''}{F'}
\end{align}
for some $b_2 \in \R$. 
This proves Theorem~\ref{thm:conditions} in the case $a=0$ by selecting 
$b_k' = b_k \cdot \left( \frac{n-1}{n-2} \right)^{k-1}$ ($0 \leq k\leq n, k\neq 2$) and $b_2' = 1 + b_2\cdot \frac{n-1}{n-1}$.

\subsection{The general (inhomogeneous) case}\label{ssec:aneq0}

\begin{proof}[Proof of Theorem~\ref{thm:conditions}]
If $a \neq 0$, 
then \eqref{eq:kuj} is an inhomogeneous linear differential equation for $f_k$, 
and by \eqref{eq:solforf_kifa=0} its general solution (by variation of parameters) is
\begin{equation} \label{eq:sol}
f_{k}=(b_{k}+B_{k}) \cdot \left( \frac{n-1}{n-2} \right)^{k-1} \cdot \abs{F}^{\frac{k-n }{n-2}} \cdot (F')^{k-1}, 
\end{equation} 
for some function $B_k=B_k(u)$ and constant $b_{k} \in \mathbb R$. 
Write $f_k$ in the form $f_k=(b_{k}+B_{k}) h_k$, 
where 
\[
h_k=\left( \frac{n-1}{n-2} \right)^{k-1} \cdot \abs{F}^{\frac{k-n }{n-2}} \cdot (F')^{k-1}.
\]
Putting \eqref{eq:sol} into \eqref{eq:kuj} we obtain  
\[
\bigl( \left( b_k+B_k \right) h_k' g+ \left( k-1 \right) \left( b_k+B_k \right) h_k g'+ \left( 2-k \right) \left( b_k+B_k \right) h_k \bigr) 
+ B_k' h_k g+ (k+1) a f_{k+1} g=0.
\]
Now, 
$h_k$ is a particular solution of the homogeneous differential equation \eqref{eq:ka0}, 
thus 
\[
(b_k+B_k) h_k' g+(k-1)(b_k+B_k) h_k g'+ (2-k)(b_k+B_k) h_k=0. 
\]
Since $h_k(u)\neq 0$ for all $u \in U$, 
$B_k$ is a particular solution of the differential equation  
\begin{equation}
\label{eq:Bk}
B_{k}' 
=- \frac{(k+1) a f_{k+1}}{h_k} 
= -(k+1) a f_{k+1} \left(\frac{n-1}{n-2}\right)^{1-k} \abs{F}^{\frac{n-k}{n-2}} (F')^{1-k},  \qquad (3 \leq k \leq n-1). 
\end{equation}
We prove by induction on $m=n-k$ that $B_{k}=A_{k}$ ($3 \leq k \leq n-1$) by showing that $B_k$ satisfies the recursive system of differential equations \eqref{eq:Aksystem} of Lemma~\ref{lem:Ak}. 
Let $m=1$, 
that is $k=n-1$. 
From \eqref{eq:F} we have $f_n=\varepsilon (F')^{n-1}$. 
Applying \eqref{eq:Bk} we obtain  
\begin{equation}
\label{eq:Bn-1der}
B_{n-1}'
=-n a f_n \left(\frac{n-1}{n-2}\right)^{2-n} (F')^{2-n} \abs{F}^{\frac{1}{n-2}}
=- n a \varepsilon \left(\frac{n-1}{n-2}\right)^{2-n} \abs{F}^{\frac{1}{n-2}} F'. 
\end{equation} 
Comparing \eqref{eq:Bn-1der} and 
\eqref{eq:Aksystem} for $m=1$,
we find $B'_{n-1}=A'_{n-1}$ by choosing $b_n = \varepsilon \left( \frac{n-1}{n-2}\right)^{1-n}$. 
Hence, a particular solution $B_{n-1}$ of the differential equation $B'_{n-1}=A'_{n-1}$ is $A_{n-1}$. 
Therefore, for $k=n-1$ we have $B_{k}=A_{k}$. 

Assume that for an integer $4 \leq k \leq n-1$
$B_{k}=A_{k}$ holds, 
thus from \eqref{eq:sol} for $k=n-m+1$ we have
\[
f_{n-m+1}=(b_{n-m+1}+ A_{n-m+1}) \cdot \left( \frac{n-1}{n-2} \right)^{n-m} \cdot \abs{F}^{\frac{1-m}{n-2}} \cdot (F')^{n-m}.
\]
Putting this into \eqref{eq:Bk} for $k=n-m$ we obtain
\begin{equation} 
\label{eq:Bn-m-1der}
\begin{aligned}
B'_{n-m}(u)
&= -(n-m+1)a f_{n-m+1} \cdot \left( \frac{n-1}{n-2} \right)^{1-n+m} \cdot \abs{F}^{\frac{m}{n-2}} \cdot (F')^{1-n+m} \\
&= -(n-m+1)a  \cdot  \frac{n-1}{n-2} \cdot \abs{F}^{\frac{1}{n-2}} \cdot (F')(b_{n-m+1}+ A_{n-m+1}). 
\end{aligned}
\end{equation} 
Comparing \eqref{eq:Aksystem} for $k=n-m$ and \eqref{eq:Bn-m-1der} we see that $B'_{n-m}=A'_{n-m}$. 
Hence a particular solution $B_{n-m}$ of the differential equation $B'_{n-m}=A'_{n-m}$ is $A_{n-m}$. 
Therefore for $k=n-m$ one has $B_{k}=A_{k}$. 
By induction, 
for all $3 \leq k \leq n-1$ one has $B_{k}=A_{k}$. 
Thus \eqref{eq:fk} holds for $3 \leq k \leq n-1$, 
and \eqref{eq:B_k} is proved. 

We continue by proving \eqref{eq:f2} and \eqref{eq:B_2}, 
that is we show the condition on $f_k$ for $k=2$. 
Now, 
$f_2$ is the solution of the inhomogeneous linear differential equation \eqref{eq:2uj}.
The general solution of \eqref{eq:2uj} by \eqref{eq:solforf_2ifa=0} and by variation of parameters has the form
\begin{equation} \label{eq:equf2}
f_2=\frac{g'+\left(b_2+B_2\right)}{g}
\end{equation} 
for some function $B_2=B_2(u)$ and constant $b_{2} \in \mathbb R$. 
Putting \eqref{eq:equf2} into \eqref{eq:2uj},
then using 
$\left(\frac{1}{g}\right)'g = -\frac{g'}{g}$,
and the fact that $\frac{g'+b_2}{g}$ is the solution to the homogeneous differential equation \eqref{eq:2a0},
one has 
\begin{align*}
0 &= -g'' + \left( \frac{g' + b_2}{g} \right)' g + B_2' + B_2 \left( \frac{1}{g} \right)'g + 3af_3g + \frac{g'+b_2}{g}g' + \frac{B_2}{g}g' 
\\
&= -g'' + \left( \frac{g' + b_2}{g} \right)' g + \frac{g'+b_2}{g}g' + B_2 \left( \left( \frac{1}{g} \right)'g + \frac{g'}{g} \right) + B_2' + 3af_3g \\
&= B_2' + 3af_3g, 
\end{align*}
that is $B_2'=-3 a f_3 g$.  Using the form of $f_3$ given by \eqref{eq:sol}
and the definition of $g$ we obtain
\begin{equation} \label{eq:b2der}
B_2'= -3 a \left(b_3 +B_3 \right) \frac{n-1}{n-2} \nu \abs{F}^{\frac{1}{n-2}} F'. 
\end{equation}  
Now, 
$A_3=B_3$ by \eqref{eq:B_k}, 
thus comparing \eqref{eq:b2der} and \eqref{eq:Aksystem} for $k=2$ yields $B_2'=\nu A_2'$. 
Hence, a particular solution $B_{2}$ of the differential equation $B'_{2}=\nu A'_{2}$ has the form $B_2= \nu A_2$.  
Thus, \eqref{eq:f2} and \eqref{eq:B_2} hold. 

Now, 
we obtain the condition on $f_k$ for $k=1$. 
The function $f_1$ is the solution of the inhomogeneous linear differential equation \eqref{eq:1uj}. 
The general solution of \eqref{eq:1uj} by \eqref{eq:solforf_kifa=0} and by variation of parameters is 
\begin{equation} \label{eq:f1gen}
f_1=\left( b_1+B_1 \right) \abs{F}^{\frac{1-n}{n-2}}
\end{equation} 
for some function $B_1=B_1(u)$ and constant $b_{1} \in \mathbb R$. 
Putting \eqref{eq:f1gen} into \eqref{eq:1uj},
and using the fact that $\abs{F}^{\frac{1-n}{n-2}}$ is the solution to the homogeneous differential equation \eqref{eq:1a0},
one obtains 
\begin{align*}
0 &= a \left( 1 - 2g' \right) + \left( b_1 + B_1 \right) \left( \abs{F}^{\frac{1-n}{n-2}} \right)' g 
+ B_1' \abs{F}^{\frac{1-n}{n-2}} g+ 2af_2g +\left( b_1+B_1 \right) \abs{F}^{\frac{1-n}{n-2}} 
\\
&= \left( b_1 + B_1 \right) \left( \left( \abs{F}^{\frac{1-n}{n-2}} \right)' g + \abs{F}^{\frac{1-n}{n-2}} \right)
+ B_1' \abs{F}^{\frac{1-n}{n-2}} g + a \left( 1 - 2g' \right) + 2af_2g \\
&= B_1' \abs{F}^{\frac{1-n}{n-2}} g+ a(1-2 g') +2 a g f_2. 
\end{align*}
As $f_2$ has the form \eqref{eq:equf2} and $B_2 = \nu A_2$ by
\eqref{eq:B_2}, we obtain that
\begin{equation}
\label{eq:b1der}
\begin{aligned}
  B_1' &= \frac{\abs{F}^{\frac{n-1}{n-2}}}{g} a \left( 2g'-1 - 2 \left( g'
      + b_2 + B_2 \right) \right) = - a \left( \frac{n-1}{n-2} \right) \nu
  \abs{F}^{\frac{1}{n-2}} F' \left( 1+ 2 b_2 + 2 \nu A_2 \right)
  \\
  &= - 2 a \left( b_2 + A_2 \right) \left( \frac{n-1}{n-2} \right)
  \abs{F}^{\frac{1}{n-2}} F' - a \left( 2b_2 \left( \nu -1 \right) + \nu
  \right) \left( \frac{n-1}{n-2} \right) \abs{F}^{\frac{1}{n-2}} F'.
\end{aligned}
\end{equation}
Comparing \eqref{eq:b1der} and \eqref{eq:Aksystem} for $k=1$ we obtain
\begin{equation} \label{eq:b1dermeg} 
B_1'=A_1' -a (2b_2(\nu -1)+ \nu) \frac{n-1}{n-2} \abs{F}^{\frac{1}{n-2}} F'. 
\end{equation} 
Hence a particular solution $B_{1}$ of the differential equation \eqref{eq:b1dermeg} has the form  
\begin{equation}\label{eq:b1meg}
B_1=A_1 -a (2b_2(1-\nu )+ 1)\abs{F}^{\frac{n-1}{n-2}}. 
\end{equation}
Therefore, 
\eqref{eq:fk} holds for $k=1$, 
and \eqref{eq:B_1} is proved.

Finally, we prove that \eqref{eq:fk} holds for $k=0$.  For $k=0$ the
function $f_0$ is the solution of the inhomogeneous linear differential
equation \eqref{eq:0uj}.  The general solution of \eqref{eq:0uj} by
\eqref{eq:solforf_kifa=0} and by variation of parameters is
\begin{equation} \label{eq:f0gen} f_0 = \left( b_0+ B_0 \right) \left(
    \frac{n-2}{n-1} \right) \abs{F}^{\frac{-n}{n-2}} \left(F'\right)^{-1}
\end{equation}
for some function $B_0=B_0(u)$ and constant $b_{0} \in \mathbb R$.  Putting
\eqref{eq:f0gen} into \eqref{eq:0uj}, and using the fact that
$\left( \frac{n-2}{n-1} \right) \abs{F}^{\frac{-n}{n-2}}
\left(F'\right)^{-1} $ is the solution to the homogeneous differential
equation \eqref{eq:0a0}, we obtain
\begin{align*}
  0 &= -a^2 g + \left( b_0+ B_0 \right) \left( \frac{n-2}{n-1} \right) \left( \abs{F}^{\frac{-n}{n-2}} \left(F'\right)^{-1} \right)' g 
      + B_0'  \left( \frac{n-2}{n-1} \right) \abs{F}^{\frac{-n}{n-2}} \left(F'\right)^{-1}g 
  \\
    & \quad+ a f_1 g 
      - \left( b_0+ B_0 \right) \left( \frac{n-2}{n-1} \right) \abs{F}^{\frac{-n}{n-2}} \left(F'\right)^{-1} g' 
      + 2 \left( b_0+ B_0 \right) \left( \frac{n-2}{n-1} \right) \abs{F}^{\frac{-n}{n-2}} \left(F'\right)^{-1}
  \\
    &= B_0'  \left( \frac{n-2}{n-1} \right) \abs{F}^{\frac{-n}{n-2}}
      \left(F'\right)^{-1}g  -a^2 g + a f_1 g. 
\end{align*}
As $f_1$ has the form \eqref{eq:f1gen},
and $B_1$ has the form \eqref{eq:b1meg}, 
we obtain 
\begin{equation} \label{eq:b0der}
\begin{aligned}
  B_0' &= \left( a^2 -a f_1 \right) \left( \frac{n-1}{n-2} \right)
  \abs{F}^{\frac{n}{n-2}} F'
  =\left(a^2-a \left( b_1+B_1 \right) \abs{F}^{\frac{1-n}{n-2}} \right) \left( \frac{n-1}{n-2} \right) \abs{F}^{\frac{n}{n-2}} F' \\
  &=\left( a^2-a \left(b_1+A_1-a \left( 2b_2 \left( 1-\nu \right) + 1 \right) \abs{F}^{\frac{n-1}{n-2}}  \right) \abs{F}^{\frac{1-n}{n-2}} \right) \left( \frac{n-1}{n-2} \right) \abs{F}^{\frac{n}{n-2}} F' \\
  &= -a \left( b_1+A_1 \right) \left( \frac{n-1}{n-2} \right) \abs{F}^{\frac{1}{n-2}} F' + a^2 \left( 2 + 2b_2 \left( 1- \nu \right) \right) \left( \frac{n-1}{n-2} \right) \abs{F}^{\frac{n}{n-2}} F' \\
  &= -a \left( b_1+A_1 \right) \left( \frac{n-1}{n-2} \right)
  \abs{F}^{\frac{1}{n-2}} F' + 2 a^2 \left( 1+ b_2 \left( 1- \nu \right)
  \right) \left( \frac{n-1}{n-2} \right) \abs{F}^{\frac{n}{n-2}} F'.
\end{aligned}
\end{equation}
Comparing \eqref{eq:b0der} and \eqref{eq:Aksystem} for $k=0$ we have 
\begin{equation} \label{eq:B0ujujder}
B_0'=A_0'+ 2 a^2 \left( 1+ b_2 \left( 1- \nu \right) \right) \left( \frac{n-1}{n-2} \right) \abs{F}^{\frac{n}{n-2}} F'. 
\end{equation} 
Therefore, 
a particular solution $B_0$ of \eqref{eq:B0ujujder} is  
\[
B_0=A_0+ a^2 (1+ b_2(1 -\nu )) \nu \abs{F}^{\frac{2(n-1)}{n-2}}. 
\]
Hence, 
\eqref{eq:fk} holds for $k=0$, 
and \eqref{eq:B_0} is proved.
This finishes the proof of Theorem~\ref{thm:conditions}
\end{proof}

\section{Open problems}\label{sec:problems}

Several questions arise after determining the symmetries of \eqref{EQ:01}. 
Indeed, 
if the Lie group of symmetries is at least two dimensional, 
then one can apply the two-dimensional solvable Lie group to obtain the solutions of \eqref{EQ:01}. 

\begin{problem}
Determine the solutions of \eqref{EQ:01}
provided $f_k$ ($0\leq k\leq n$, $n\geq 4$) satisfy the conditions of Theorem~\ref{thm:main}.
\end{problem}

The only remaining case for \eqref{EQ:01} not covered by Theorem~\ref{thm:main} or by \cite{PaliathanasisLeach2016, Pandey2009-1, Pandey2009-2, Tiwari2013, Tiwari2014, Tiwari2015} is when $n=3$. 
Then one cannot immediately conclude $\xi_u=0$ from \eqref{EQ:0}, 
because the $\dot{u}^4$ term of \eqref{EQ:0} is identically 0. 
In fact, 
for $n=3$ the symmetry condition translates to
\begin{equation}\label{eq:n=3}
\begin{alignedat}{5}
  -\eta_{tt} + f_0'\eta + f_1\eta_t - f_0 \eta_u + 2 f_0 \xi_t & =0,
  \\
  \left( \xi_{tt} - 2 \eta_{tu} \right) + f_1'\eta + 2 f_2 \eta_t +
  f_1\xi_t + 3 f_0 \xi_u & =0,
  \\
  \left( 2\xi_{tu}-\eta_{uu} \right) + f_2' \eta + 3 f_3 \eta_t + f_2
  \eta_u + 2 f_1 \xi_{u} & =0,
  \\
  \xi_{uu} + f_3' \eta + 2 f_3 \eta_u - f_3 \xi_t + f_2 \xi_u & =0.
\end{alignedat}
\end{equation}
A potential simplification of the system \eqref{eq:n=3} might be to eliminate $f_2$ from \eqref{eq:n=3}
by using a coordinate change $v=G(u)$, 
for some bijective, two-times differentiable $G$, 
for which $G''(u) = G'(u)f_2(u)$ is satisfied
(see e.g.~\cite{PaliathanasisLeach2016}).
This, however, 
still does not give an immediate answer as to what the solutions of \eqref{eq:n=3} are.

\begin{problem}
Determine all symmetries (and solutions) of the autonomous differential equation 
\[
\ddot{u} = f_0 \left( u \right) + \dot{u} f_1 \left(u \right) + \dot{u}^2 f_2 \left( u \right)  + \dot{u}^3 f_3 \left( u \right), 
\]
where $f_0, f_1, f_2, f_3$ are arbitrary continuous functions in $u$. 
\end{problem}

\def\cprime{$'$}

\vspace{1cm}


\begin{thebibliography}{10}

\bibitem{vinogradov}
A.~V. Bocharov, V.~N. Chetverikov, S.~V. Duzhin, N.~G. Khor{\cprime}kova, I.~S.
  Krasil{\cprime}shchik, A.~V. Samokhin, Y.~N. Torkhov, A.~M. Verbovetsky, and
  A.~M. Vinogradov.
\newblock {\em Symmetries and conservation laws for differential equations of
  mathematical physics}, volume 182 of {\em Translations of Mathematical
  Monographs}.
\newblock American Mathematical Society, Providence, RI, 1999.
\newblock Edited and with a preface by Krasil{\cprime}shchik and Vinogradov,
  Translated from the 1997 Russian original by Verbovetsky [A. M.
  Verbovetski{\u\i}] and Krasil{\cprime}shchik.

\bibitem{Nucci1Biology}
M.~Edwards and M.~C. Nucci.
\newblock Application of {L}ie group analysis to a core group model for
  sexually transmitted diseases.
\newblock {\em J. Nonlinear Math. Phys.}, 13(2):211--230, 2006.

\bibitem{Faria2008}
J.~R. Faria, J.~C. Cuestas, and L.~A. Gil-Alana.
\newblock {\em Unemployment and entrepreneurship: a cyclical relation?}
\newblock Discussion papers. Nottingham Trent University, Bottingham Business
  School, Economics Division, 2008/2.

\bibitem{Fitzhugh1961}
R.~Fitzhugh.
\newblock Impulses and physiological states in theoretical models of nerve
  membrane.
\newblock {\em Biophys J.}, 1(6):445--466, 1961.

\bibitem{Goodwin1975}
R.~M. Goodwin.
\newblock A growth cycle.
\newblock In C.~H. Feinstein, editor, {\em Socialism, Capitalism and Economic
  Growth}, pages 54--58, Cambridge, 1975. Cambridge University Press.

\bibitem{Jones}
D.~S. Jones, M.~Plank, and B.~D. Sleeman.
\newblock {\em Differential Equations and Mathematical Biology}.
\newblock Chapman and Hall/CRC, second edition, 2009.

\bibitem{Kaldor1940}
N.~Kaldor.
\newblock A model of the trade cycle.
\newblock {\em The Economic Journal}, 50(197):78--92, 1940.

\bibitem{Kalecki1937}
M.~Kalecki.
\newblock A theory of the business cycle.
\newblock {\em The Review of Economic Studies}, 4(2):77--97, 1937.

\bibitem{Levinson1943}
N.~{Levinson}.
\newblock {On the existence of periodic solutions for second order differential
  equations with a forcing term.}
\newblock {\em {J. Math. Phys., Mass. Inst. Techn.}}, 22:41--48, 1943.

\bibitem{LevinsonSmith1942}
N.~{Levinson} and O.~K. {Smith}.
\newblock {A general equation for relaxation oscillations.}
\newblock {\em {Duke Math. J.}}, 9:382--403, 1942.

\bibitem{Mickens1981}
R.~E. {Mickens}.
\newblock {\em {An {I}ntroduction to {N}onlinear {O}scillations.}}
\newblock {Cambridge etc.: Cambridge University Press. XIV, 224 p.}, 1981.

\bibitem{Nagumo1962}
J.~Nagumo, S.~Arimoto, and S.~Yoshizawa.
\newblock An active pulse transmission line simulating nerve axon.
\newblock {\em Proc. IRE.}, 50:2061--2070, 1962.

\bibitem{Nucci}
M.~C. Nucci.
\newblock The role of symmetries in solving differential equations.
\newblock {\em Mathl. Comput. Modelling}, 25(8--9):181--193, 1997.

\bibitem{Nucci2Biology}
M.~C. Nucci and P.~G.~L. Leach.
\newblock Singularity and symmetry analyses of mathematical models of
  epidemics.
\newblock {\em South African Journal of Science}, 105:136--146, 2009.

\bibitem{NucciSanchini2015}
M.~C. Nucci and G.~Sanchini.
\newblock Symmetries, {L}agrangians and {C}onservation {L}aws of an {E}aster
  {I}sland {P}opulation {M}odel.
\newblock {\em Symmetry}, 7(3):1613--1632, 2015.

\bibitem{Olver}
P.~J. Olver.
\newblock {\em Applications of {L}ie {G}roups to {D}ifferential {E}quations},
  volume 107 of {\em Graduate Texts in Mathematics}.
\newblock Springer-Verlag, New York, second edition, 1993.

\bibitem{PaliathanasisLeach2016}
A.~{Paliathanasis} and P.~G.~L. {Leach}.
\newblock {Comment on ``Classification of Lie point symmetries for quadratic
  Li\'enard type equation $\ddot{x} + f(x) \dot{x}^2 + g(x) = 0$'' [J. Math.
  Phys. 54, 053506 (2013)] and its erratum [J. Math. Phys. 55, 059901 (2014)].}
\newblock {\em {J. Math. Phys.}}, 57(2):024101, 2, 2016.

\bibitem{Pandey2009-1}
S.~N. {Pandey}, P.~S. {Bindu}, M.~{Senthilvelan}, and M.~{Lakshmanan}.
\newblock {A group theoretical identification of integrable cases of the
  Li{\'{e}}nard-type equation $\ddot x +f(x)\dot x +g(x)=0$. I: Equations
  having nonmaximal number of Lie point symmetries.}
\newblock {\em {J. Math. Phys.}}, 50(8):082702, 19, 2009.

\bibitem{Pandey2009-2}
S.~N. {Pandey}, P.~S. {Bindu}, M.~{Senthilvelan}, and M.~{Lakshmanan}.
\newblock {A group theoretical identification of integrable equations in the
  Li\'enard-type equation $\ddot x+f(x)\dot x+g(x) = 0$. II: Equations having
  maximal Lie point symmetries.}
\newblock {\em {J. Math. Phys.}}, 50(10):102701, 25, 2009.

\bibitem{HandbookODE}
A.~D. {Polyanin} and V.~F. {Zaitsev}.
\newblock {\em {Handbook of {E}xact {S}olutions for {O}rdinary {D}ifferential
  {E}quations. 2nd. ed.}}
\newblock Boca Raton, FL: CRC Press, 2nd. ed. edition, 2003.

\bibitem{Rayleigh1945}
J.~W. {Strutt}.
\newblock {\em {The {T}heory of {S}ound. 2nd ed.}}
\newblock {New York: Dover Publications. Two volumes in one. XIII, 480 p.; XII,
  504 p.}, 1945.

\bibitem{Tiwari2013}
A.~K. {Tiwari}, S.~N. {Pandey}, M.~{Senthilvelan}, and M.~{Lakshmanan}.
\newblock {Classification of Lie point symmetries for quadratic Li\'enard type
  equation $\ddot{x}+f(x)\dot{x}^2+g(x)=0$.}
\newblock {\em {J. Math. Phys.}}, 54(5):053506, 19, 2013.

\bibitem{Tiwari2014}
A.~K. {Tiwari}, S.~N. {Pandey}, M.~{Senthilvelan}, and M.~{Lakshmanan}.
\newblock {Erratum: ``Classification of Lie point symmetries for quadratic
  Li\'enard type equation $\ddot{x}+f(x)\dot{x}^2+g(x)=0$''.}
\newblock {\em {J. Math. Phys.}}, 55(5):059901, 2, 2014.

\bibitem{Tiwari2015}
A.~K. Tiwari, S.~N. Pandey, M.~Senthilvelan, and M.~Lakshmanan.
\newblock {Lie point symmetries classification of the mixed Li{\'e}nard-type
  equation}.
\newblock {\em Nonlinear Dynamics}, 82(4):1953--1968, 2015.

\bibitem{vanderPol1927}
B.~van~der Pol.
\newblock On relaxation-oscillations.
\newblock {\em The Lond, Edinburgh, and Dublin Philosophical Magazine and
  Journal of Science}, 2:978--992, 1927.

\bibitem{vanderPolvanderMark1928}
B.~van~der Pol and J.~van~der Mark.
\newblock The heartbeat considered as a relaxation oscillation, and an
  electrical model of the heart.
\newblock {\em The Lond, Edinburgh, and Dublin Philosophical Magazine and
  Journal of Science}, 6:763--775, 1928.

\end{thebibliography}
\end{document}